\documentclass[12pt]{article}

\usepackage[english]{babel}
\usepackage[utf8]{inputenc}
\usepackage[T1]{fontenc}
\usepackage{amsmath}
\usepackage{amssymb}
\usepackage{amsthm}
\usepackage{cite}
\usepackage{caption}

\usepackage{tikz}

\tikzstyle{sommetb}=[circle,draw,scale=0.5,fill=black]
\tikzstyle{sommeti}=[circle,draw,scale=0.3,fill=black]

\title{Lower bounds for the first eigenvalue of the Steklov problem on graphs}
\author{Hélène Perrin}
\date{}

\newtheorem{defi}{Definition}
\newtheorem{note}{Remark}
\newtheorem{lem}{Lemma}
\newtheorem{prop}{Proposition}
\newtheorem{exa}{Example}
\newtheorem{thm}{Theorem}

\begin{document}
\maketitle

\begin{abstract}
We give lower bounds for the first non-zero Steklov eigenvalue on connected graphs. These bounds depend on the extrinsic diameter of the boundary and not on the diameter of the graph. We obtain a lower bound which is sharp when the cardinal of the boundary is $2$, and asymptotically sharp as the diameter of the boundary tends to infinity in the other cases. We also investigate the case of weigthed graphs and compare our result to the Cheeger inequality.
\end{abstract}

\section{Introduction}
The Steklov problem on a compact Riemannian manifold $M$ with boundary $\partial M$ is known to be
\[
\begin{cases}
\triangle u =0&\text{in } M\\
\frac{\partial u}{\partial n}=\sigma u &\text{on } \delta M
\end{cases}
\]
where $\triangle$ is the Laplace-Beltrami operator and $\frac{\partial u}{\partial n}$ is the outward normal derivative along the boundary $\delta M$. It is a classical result that if the boundary is sufficiently regular, the spectrum of the Steklov problem is discrete and its eigenvalues form a sequence $0=\sigma_0<\sigma_1\leq\sigma_2\leq \dots \nearrow \infty$. It is also known that this spectrum coincides with the one of the Dirichlet-to-Neumann operator.

The Steklov problem has been pointed out to be of particular interest for spectral geometry because it has features that distinguish it from the Dirichlet and Neumann problems for the Laplacian (see \cite{gpsurvey}). For the Laplacian, it is classical to try to understand spectral properties of a manifold via a discretization (see e.g. \cite{ChavIneq} and \cite{mantuano2005discretization}). The eigenvalues of the Laplacian on $M$ are related to those of its discretization and the discrete Laplacian is studied in order to get information about the spectral geometry of $M$. This approach has been recently done for the Steklov operator in \cite{colbois2016steklov}. In consequence, it becomes very relevant to study the spectrum of the discrete Steklov operator.

The object of this paper is to study lower bounds for the first non-zero eigenvalue of the Steklov problem on graphs. The study of the spectral geometry associated to this eigenvalue has been initiated in a recent article by Hua, Huang and Wang \cite{hua2017}. They define two Cheeger-type constants for the Steklov problem on graphs, based on the constants by Jammes and Escobar already existing for the continuous case, and give two very interesting lower bound estimates depending on these constants. In the case of the so called Jammes-type Cheeger estimate, there is not a Buser-type inequality as for the combinatorial Laplacian, that is, often, the Jammes-type Cheeger constant is small while the first eigenvalue is not (a typical example of this phenomenon is given in Example \ref{D_{n+3}}).

In order to better understand the geometry of the graph captured by $\sigma_1$, we investigate a different lower bound, depending on geometric features of the boundary. In particular, it depends on the extrinsic diameter of the boundary, but the diameter of the whole graph is not involved. It is sharp when there are only two vertices in the boundary and, otherwise, it becomes asymptotically sharp as the diameter of the boundary goes to infinity. Our bound is, in some cases, more accurate than the Jammes-type Cheeger estimate of Hua et al. We draw here the reader's attention to the fact that, while Hua et al. are considering weighted graphs, we focus in this paper on graphs with weight one on all the edges, which are relevant in the context of discretization. The definition of the Steklov problem, or Dirichlet-to-Neumann operator, is also slightly different. In analogy to the Laplacian, we could say that we use a non-normalized form of the Steklov problem on graphs (for a discussion about the several definitions of the combinatorial Laplacian, see e.g. \cite{ButlerPHD}). However, in the examples that we give for comparison, the two definitions coincide and in our last section we will adapt a part of our result to the settings of the article by Hua et al.

The paper is structured as follows: in section 2, we give the definitions necesary to state the Steklov problem on graphs; the main lower bound results, Theorem \ref{thm1} and Theorem \ref{thm2}, are given in section 3, respectively in part I and II of this section; in the last section, we extend the first result to weighted graphs. Theorem \ref{thm1} can be deduced from Theorem \ref{thm2}, but we begin with it since the proof is very simple and we use it as a starting point for the more technical proof of Theorem \ref{thm2}.

This article resulted from a master's thesis realized in 2017 at the University of Neuchâtel under the direction of the Professor Bruno Colbois.

\section{Preliminaries}

In a graph $\Gamma=(V,E)$, for a subset $S\subset V$, the boundary of $S$ in $G$ is defined in the following way. The edge boundary of $S$ is $\partial S:=E(S,S^c)$, where $E(\Omega_1$, $\Omega_2)$ is the set of edges between the two subsets $\Omega_1$, $\Omega_2\subset V$, i.e. the set $\{e=\{i,j\}\in E|i\in \Omega_1\text{, }j\in\Omega_2\}$. The vertex boundary of $S$ is $\delta S:=\{i\in S^c|i\sim j \text{ for some } j \in S\}$. We use $i\sim j$ to signify that $\{i,j\}\in E$. The degree of a vertex $i$ is denoted $d(i)$.

If we forget the vertices that are not in in $S\cup \delta S$, and the edges that have no endpoint in $S$, we get what we call a graph with boundary.

\begin{defi}
A graph with boundary is a pair $(\Gamma, B)$, where $\Gamma=(V,E)$ is a simple graph, that is, without loops and multiple edges, and $B\subset V$ is a subset of $V$ such that $\delta( B^c)=B$ and $E(B,B)=\emptyset$. We call $B$ the boundary and $B^c:=I$ the interior of the graph.
\end{defi}

\begin{note} If $\Gamma=(V,E)$ is a simple graph and $\Omega$ a subset of $V$, then $\widetilde{\Omega}:=(\bar{\Omega},E(\Omega,\bar{\Omega}))$ where $\bar{\Omega}:=\Omega\cup\delta\Omega$ is a graph with boundary according to our definition.
\end{note}

In this paper, we always consider graphs with boundary, connected, and finite. The space of all real functions defined on the vertices $V$, denoted $\mathbb{R}^V$, is the Euclidean space of dimension $|V|$. Similarly, the space of real functions defined on the vertices of the boundary, denoted $\mathbb{R}^{B}$, is the Euclidean space of dimension $|B|$. The scalar product in an Euclidean space is denoted by $\langle\cdot,\cdot\rangle$. In addition to the scalar product, we introduce on $\mathbb{R}^V$ the following bilinear form: for functions $v,w\in \mathbb{R}^V$, we define $\langle v,w\rangle_B:=\langle v|_{B}, w|_{B}\rangle$, we write $\|v\|_B$ for $\langle v,v\rangle_B^{1/2}$, and $v\perp_B w$ if $\langle v,w\rangle_B=0$.

The Laplacian matrix of a graph $\Gamma$, denoted $\triangle$, is
\[(\triangle)_{ij}=
\begin{cases}
d(i) &\text{if }i=j\\
-1 &\text{if i and j are adjacent}\\
0 &\text{otherwise}
\end{cases}
\]
A function $v=(v_1,...,v_{|V|})$ on $(\Gamma,B)$ is called harmonic if
\[
(\triangle v)_i=\sum_{j\sim i}(v_i-v_j)=0\quad \forall i \in I
\]
The normal derivative of $v$, denoted $\frac{\partial v}{\partial n}$, satisfies
\[(\frac{\partial v}{\partial n})_i=\sum_{j\in I,j\sim i}(v_i-v_j) \quad i \in B\]

Since $E(B,B)=\emptyset$, we have $(\frac{\partial v}{\partial n})_i=\sum_{j\in I,j\sim i}(v_i-v_j)=(\triangle v)_i$ for $i\in B$. In analogy to the Riemannian case, we give the following definition:

\begin{defi}The Steklov problem on a graph with boundary is
\[(\triangle v)_i=
\begin{cases}
0 &\text{if }i\in I\\
\sigma v_i &\text{if }i \in B\\
\end{cases}
\]
\end{defi}

The eigenvalues of this problem are the same as the eigenvalues of the discrete Dirichlet-to-Neumann operator defined by \cite{hua2017} in the case that the weights of the edges are all $1$ and the degree of the boundary vertices is $1$. In order to see this, we need this useful linear algebra lemma:

\begin{lem} For a connected graph with boundary $(\Gamma,B)$, given any $\varphi \in \mathbb{R}^{B}$, there is a unique function $\tilde{\varphi}\in\mathbb{R}^V$, called the harmonic extension of $\varphi$, which satisfies
\begin{align}
(\triangle\tilde{\varphi})_i=0 \qquad \text{if } i\in I\\
\tilde{\varphi}_i=v_i \qquad \text{if } i\in B
\end{align}
\end{lem}
\begin{proof}
Putting $(2)$ in $(1)$ and developing, we obtain the following linear system of equations
\[
([\triangle]_I )\tilde{\varphi}{\big|_I})_i=\sum_{j \sim i, j\in B}\varphi_j\quad\forall i\in I
\]
where $[\triangle]_I$ denotes the principal minor of $\triangle$ with rows and columns corresponding to the vertices of the interior. If the graph is connected, this matrix is invertible and thus there is a unique solution of the system.
\end{proof}

The Dirichlet-to-Neumann operator $\Lambda:\mathbb{R}^{B}\rightarrow\mathbb{R}^{B}$ maps $\varphi$ to $\Lambda \varphi:=\frac{\partial \tilde{\varphi}}{\partial n}$. Let $\sigma$ be an eigenvalue of the DtN operator and $\varphi$ an associated eigenfunction, then it is clear that the pair $(\sigma,\tilde{\varphi})$ is solution of the Steklov problem. In the other direction, if $(\sigma,v)$ is solution of the Steklov problem, it is also clear that $\sigma$ is an eigenvalue of the DtN operator, asociated to $v|_{B})$. Thus, the spectra are equivalent.

This problem has $b$ solutions $\sigma_0\leq\sigma_1\leq...\leq\sigma_{b-1}$, called the eigenvalues, and there exist $b$ associated eigenfunctions $v^0,v^1,...,v^{b-1}$ which are mutually orthogonal for the bilinear form $\langle \cdot,\cdot\rangle_B$ and can be choosen such that $\|v^i\|_B=1$. This affirmation results from the fact that $\Lambda$ is a nonegative self-adjoint operator as explained in \cite{hua2017}. A proof using basic analysis and linear algebra tools is also given in \cite{PerrinH}. The Rayleigh quotient associated to $\Lambda$ is, for a function $v\in \mathbb{R}^V$,
\[
\frac{\langle v,\triangle v\rangle}{\left\|v\right\|_B^2}=\frac{\sum_{i\sim j}(v_i-v_j)^2}{\sum_{i\in B}v_i^2}
\]
and there are variational characterizations for the eigenvalues:
\begin{align}
\sigma_j=\min_E \max_{v\in E,v\not =0}\left\{\frac{\langle v,\triangle v\rangle}{\left\|v\right\|_B^2}\right\}
\label{sigmak}
\end{align}
where $E$ is the set of all linear subspaces of $\mathbb{R}^V$ of dimension $j+1$. For $\sigma_1$, we have
\begin{align}
\sigma_1=\min_{v\in \mathbb{R}^{V},v\in [v^0]^{\perp_B}}\left\{\frac{\langle v,\triangle v\rangle}{\left\|v\right\|_B^2}\right\}
\label{carvar1}
\end{align}
It is easy to see that a constant function is an eigenfunction associated to the eigenvalue $0$. This allows us to rewrite equation (\ref{carvar1})
\begin{align}
\sigma_1=\min_{v\in \mathbb{R}^{V}}\left\{\sum_{i\sim j}(v_i-v_j)^2, \sum_{i\in B}v_i^2=1, \sum_{i\in B}v_i=0\right\}
\label{carvar}
\end{align}

\begin{note}
\label{comp}
From equation (\ref{sigmak}), we see that if $(\Gamma,B)$ is a graph with boundary, then $\sigma_k(\Gamma,B)\geq \lambda_k(\Gamma)$ where $\lambda_k(\Gamma)$ is the $k$th-eigenvalue of the combinatorial Laplacian on $\Gamma$.
\end{note}

\section{Lower bound for $\sigma_1$}

Let $(\Gamma,B)$ be a graph with boundary. The distance between two vertices $i$ and $j$ is the number of edges in the shortest path joining $i$ and $j$. We will denote it $d(i,j)$.

\begin{defi} The diameter $d$ of $(\Gamma,B)$ is the maximum distance between any two vertices of $(\Gamma,B)$, i.e. $d=\max\{d(i,j)|i,j\in V\}$.

\end{defi}

\begin{defi} The diameter of the boundary $d_{B}$ is the extrinsic diameter of $B$ in $\Gamma$ or, in other words, it is the maximum distance between any two vertices of $B$, i.e. $d_{B}=\max\{d(i,j)|i,j\in B\}$.
\end{defi}

\subsection{Lower bound for $\sigma_1$ (I)}
We give now a first lower bound in terms of the diameter of the boundary and of the number of vertices of the boundary.

\begin{thm} Let $(\Gamma,B)$ be a connected graph with diameter of the boundary $d_{B}$ and $|B|=b$. We have
\[
\sigma_1\geq \frac{b}{(b-1)^2\cdot d_{B}}
\]
The bound is optimal when $b=2$.
\label{thm1}
\end{thm}

\begin{proof} Suppose $v$ is an eigenfunction achieving $\sigma_1$ in (\ref{carvar}).
Let $\alpha$ be the vertex such that $|v_\alpha|=\max_{i}|v_i|$. We can suppose that $v_\alpha$ is positive because otherwise, we can take $-v$ that also achieves $\sigma_1$ in (\ref{carvar})
and it will be the case. Let also $\beta$ be the vertex such that $v_\beta=\min_{i}v_i$.
From $\sum_{i\in B}v_i^2=1$, we get
\[
1\leq\sum_{i\in B}v_\alpha^2\Rightarrow v_\alpha\geq\frac{1}{\sqrt{b}}
\]
and from $\sum_{i\in B}v_i=0$, we get

\[
-\frac{1}{\sqrt{b}}\geq -v_\alpha=\sum_{i\in B, i\not =\alpha}v_i\geq(b-1)v_\beta
\Rightarrow v_\beta\leq-\frac{1}{(b-1)\sqrt{b}}
\]
so finally we have
\[
v_\alpha-v_\beta \geq \frac{b}{(b-1)\sqrt{b}}
\]

Since the graph is connected and the diameter of the boundary is $d_{B}$, there exists a path of length $c\leq d_{B}$ joining $\alpha$ and $\beta$. We label the $c+1$ vertices of the path by $1,...,c+1$, with $1=\alpha$ and $c+1=\beta$. Using Cauchy-Schwarz inequality, we obtain

\begin{align*}
\sigma_1&=\sum_{i\sim j}(v_i-v_j)^2\geq \sum_{i=1}^{c}(v_i-v_{i+1})^2\\
&\geq \frac{(v_\alpha-v_\beta)^2}{c}\geq \frac{1}{d_{B}}(\frac{b}{(b-1)\sqrt{b}})^2\\
&= \frac{b}{(b-1)^2\cdot d_{B}}
\end{align*}

Furthermore, using (\ref{carvar}), it is easy to compute that $\frac{2}{d_B}$ is the first non-zero eigenvalue of the graph $(P_n,B_{P_n})$ where $P_n$ is the path of length $n$ and $B_{P_n}$ the two extremities of the path. So the bound is optimal when $b=2$.
\end{proof}

\begin{note} This result is analog to Lemma 1.9. for the combinatorial normalized Laplacian in \cite{Chung95}.
\end{note}

As we will see in the following examples, this lower bound reflects other aspects of the geometrical meaning of $\sigma_1$ than the one given by Theorem 1.3 of \cite{hua2017}. This is due to the fact that the diameter of the boundary does not depend on the total number of vertices of the graph.

\begin{exa} We consider the family of graphs $\{(D_{n+3}, B_{D_{n+3}})\}_{n\in\mathbb{N}}$ as shown in Figure \ref{D_{n+3}}, which have two boundary vertices (the two bigger vertices).

\begin{figure}[h]
\begin{center}
\begin{tikzpicture}
\draw(0,0)node[sommetb]{};
\draw(0,-1)node[sommeti]{};
\draw(0,-2)node[sommetb]{};
\draw(1,-1)node[sommeti]{};
\draw(2,-1)node[sommeti]{};
\draw(4,-1)node[sommeti]{};
\draw(5,-1)node[sommeti]{};
\draw (0,0) -- (0,-1)--(0,-2);
\draw (0,-1)--(1,-1)--(2,-1);
\draw (4,-1)--(5,-1);
\draw (2,-1)--(2.5,-1);
\draw (3.5,-1)--(4,-1);
\draw[dashed] (2.5,-1)--(4,-1);
\end{tikzpicture}
\end{center}
\caption{}
\label{D_{n+3}}
\end{figure}
If we extend an eigenfunction of $(P_2,B_{P_2})$ associated to $\sigma_1$ to a function on $(D_{n+3},B_{D_{n+3}})$ by giving the value of the interior vertex of $(P_2,B_{P_2})$ to all other vertices of the interior, the Rayleigh quotient remains unchanged and it is clear that $\sigma_1(D_{n+3},B_{D_{n+3}})=\frac{2}{d_B}$. We note here that the minimizer for $\sigma_1$ if the number of vertices of the boundary and the diameter of the boundary are fixed is not unique.

By computation, we obtain that for this family of graphs, the lower bound of Theorem 1.3 in \cite{hua2017} tends to $0$ as $n$ goes to infinity. So, in this case where the diameter is unbounded but the diameter of the boundary is bounded, the Jammes-type Cheeger estimate fails to see that $\sigma_1$ is bounded.
\label{exT}
\end{exa}

The contrary happens in the next example:

\begin{exa} Let $\{\Gamma_n\}_{n\in\mathbb{N}}$ be a family of expanders (on expanders, see e.g. \cite{Chung95}). By choosing a pair of vertices $B_{\Gamma_n}:=\{i,j\}$ in $(\Gamma_n)$ that are at distance $n$, we obtain the family of graphs with boundary $(\Gamma_n,B_{\Gamma_n})_{n\in\mathbb{N}}$.
From note \ref{comp}, we deduce that $\sigma_1$ is bounded below. On this example, the Cheeger constant is better than our bound which tends to $0$ as $n$ goes to infinity.
\end{exa}

\subsection{Lower bound for $\sigma_1$ (II)}

In this part, we prove the following improvement of Theorem \ref{thm1}:


\begin{thm}
Let $(\Gamma,B)$ be a connected graph with diameter of the boundary $d_{B}$ and $|B|=b$. We have
\begin{equation}
\sigma_1\geq \frac{b}{\left\lfloor \frac{b}{2}\right\rfloor \left\lceil \frac{b}{2}\right\rceil\cdot d_B}
\label{eq1}
\end{equation}
Moreover, the bound is sharp for any $b$ as $d_B\rightarrow \infty$.
\label{thm2}
\end{thm}
We first prove (\ref{eq1}) and secondly the sharpness result. 
\begin{proof}[Proof of (\ref{eq1})]
In the proof of Theorem \ref{thm1}, we approximate the difference between the largest and the smallest values on the boundary of an eigenfunction for $\sigma_1$ using the following bound: if $v_{\alpha}$ and $v_{\beta}$ are respectively the largest and the smallest value on the boundary of an eigenfunction achieving $\sigma_1$ in (\ref{carvar}), we have $v_{\alpha}-v_{\beta}\geq \frac{b}{\sqrt{b}(b-1)}$. However, this bound is not sharp as soon as $b>2$. We improve it by solving a constrained optimization problem.

Consider the map $f_b: \mathbb{R}^b\rightarrow \mathbb{R}: (x_1,...,x_b) \mapsto x_1-x_b$. Define $D_b:=\{x \in \mathbb{R}^b: x_1\geq x_2\geq ...\geq x_b\}$ and $S_b:=\{\sum_{i=1}^{b}x_i^2=1\}\cap\{\sum_{i=1}^{b}x_i=0\}$.

\begin{prop} Let $f_b$, $D_b$, and $S_b$ be as defined above, then
\[
\min_{x \in D_b\cap S_b}f_b(x)=\frac{\sqrt{b}}{\sqrt{\left\lfloor{\frac{b}{2}}\right\rfloor}\sqrt{\left\lceil{\frac{b}{2}}\right\rceil}}
\]
\label{prop1}
\end{prop}

Now the proof of (\ref{eq1}) is immediate by replacing in the proof of Theorem \ref{thm1} the inequality $v_{\alpha}-v_{\beta}\geq \frac{b}{\sqrt{b}(b-1)}$ by $v_{\alpha}-v_{\beta}\geq\frac{\sqrt{b}}{\sqrt{\left\lfloor{\frac{b}{2}}\right\rfloor}\sqrt{\left\lceil{\frac{b}{2}}\right\rceil}}$.
\end{proof}

\begin{proof}[Proof of Proposition \ref{prop1}] The proof is by induction over the number $b$ of vertices in the boundary and uses the method of Lagrange multipliers.

\begin{paragraph}{Base case} For $b=2$, we have to show that $min_{x \in D_2\cap S_2}f_2(x)=\sqrt{2}$. Since $S_2\cap D_2=\{(\frac{1}{\sqrt{2}},\frac{-1}{\sqrt{2}})\}$, $x=(\frac{1}{\sqrt{2}},\frac{-1}{\sqrt{2}})$ achieves the minimum and it is $\sqrt{2}$.
\end{paragraph}

\begin{paragraph}{Inductive step} We show now that if the result holds for $b-1$, it holds for $b$.

We define $g_b(x)=\sum_{i=1}^{b}x_i^2-1$, $\tilde{g}_b(x)=\sum_{i=1}^{b}x_i$, and $G_b(x)=(g_b(x),\tilde{g}_b(x))$. Therefore, we have $G_b^{-1}(0)=\{\sum_{i=1}^{b}x_i^2=1\}\cap\{\sum_{i=1}^{b}x_i=0\}=S_b$.

We make now a partition of $D_b$. In order to do this, let us consider the set $M:=\{1,2,..,b-1\}$ and $M_j\subset \mathcal P (M)$ the family of sets over $M$ that have cardinality $j$, $M_0=\{\emptyset\}$. 
For a set $m_j$ in the family $M_j$, we define $E_{m_j}:=\{x \in \mathbb{R}^b: x_i>x_{i+1} \text{ if } i \not \in m_j; x_i=x_{i+1} \text{ if } i \in m_j\}$. For instance, if $b=8$ and $m_j=\{1,6,7\}\in M_3$, $E_{\{1,5,6\}}=\{x\in \mathbb{R}^b: x_1=x_2<x_3<x_4<x_5=x_6=x_7<x_8\}$.
We have the following partition of $D_b$
\[
D_b=\bigcup_{j=0}^{b-1}\bigcup_{m_j\in M_j} E_{m_j}
\]

We remark that $E_{m_j}$ is an open subset of the linear subspace of $\mathbb{R}^b$ $\{x \in \mathbb{R}^b: x_i=x_{i+1} \text{ if } i \in m_j\}$ which is of dimension $b-j$. We will say that $E_{m_j}$ is of dimension $b-j$.

On each $E_{m_j}$ intersected with $S_b$, we will look for a minimum of $f_b$. Since the minimum of $f_b$ over $D_b\cap S_b$ must be one of them, we will compare them and find it. We divide these subsets in three categories: the one of dimension $b-j=1$, those of dimension $b-j=2$ and those of dimension $b-j>2$.

\subparagraph{Case 1: $\mathbf{j=b-1}$}
The only subset of our partition of $D_b$ of dimension $1$ is $\{x\in \mathbb{R}^b:x_1=...=x_b\}$ and its intersection with $S_b$ is empty.

\subparagraph{Case 2: $\mathbf{j=b-2}$} Since the sets in the family $M_2$ only contain one element between $1$ and $b-1$, we denote the subsets of dimension $2$ $E_k:=\{x \in \mathbb{R}^b: x_1=...=x_{b-k}>x_{b-k+1}=...=x_b\}$ for $k=1,...,b-1$.

The intersection $E_k\cap S_b$ contains only one vector $y^k=(y_1,...,y_b)$ with $y_i=\frac{k\sqrt{b-k}}{(b-k)\sqrt{b}\sqrt{k}}$ if $1\leq i \leq (b-k)$ and $y_i=\frac{-\sqrt{b-k}}{\sqrt{b}\sqrt{k}}$ if $(b-k)<i\leq b$. Hence, the values possible for the minimum are $f_b(y^k)=\frac{k\sqrt{b-k}}{(b-k)\sqrt{b}\sqrt{k}} -  \frac{-\sqrt{b-k}}{\sqrt{b}\sqrt{k}}=\frac{\sqrt{b}}{\sqrt{b-k}\sqrt{k}}$ for $k=1,...,b-1$.

In order to see for which $k$ the value $f_b(y^k)$ is minimal, we can study the function $h:]0,b[\rightarrow \mathbb{R}^{+}: k \mapsto h(k)= \frac{\sqrt{b}}{\sqrt{b-k}\sqrt{b}}$. We obtain that the better $k$ is $k=\left\lfloor\frac{b}{2}\right\rfloor$ and $f_b(y^{\left\lfloor\frac{b}{2}\right\rfloor})=\frac{\sqrt{b}}{\sqrt{\left\lfloor\frac{b}{2}\right\rfloor}\sqrt{\left\lceil\frac{b}{2}\right\rceil}}$.

\subparagraph{Case 3: $\mathbf{j<b-2}$} 
It remains to investigate if there is a better minimum on the subsets of dimension greater than $2$ intersected with $S_b$. This is the most technical part of the proof. We will show that in this case, if there is a local extremum at $y=(y_1,...,y_b)$, then one of the coefficients of $y$ must be $0$ and this allows us to reduce to a situation with $b-1$ vertices.

Let $m_j\in M_j$ for $j<b-2$ and $E_{m_j}$ be its corresponding subset of $D_b$.

If $j>0$, we number the elements in $m_j$ such that to each $k\in \{1,...,j\}$ corresponds an element i:=n(k) in $m_j$. We define $h_k(x)=(x_{n(k)}-x_{n(k)+1})$ for each $k\in \{1,...,j\}$ and $H(x)=(h_1(x),...,h_j(x))$. We have $\{x \in \mathbb{R}^b: x_i=x_{i+1} \text{ if } i \in m_j\}=H^{-1}(0)$. We note that if $a$ is a minimum of $f_b$ on $E_{m_j}$, it is a relative minimum on $H^{-1}(0)$. Although we write it without indice, the function $H$ depends on the set that we are looking at. We set $K(x):=(G_b(x),H(x))$. Then $K^{-1}(0)=G_b^{-1}(0)\cap H^{-1}(0)$.

If $j=0$, we define $K(x)=G_b(x)$.

We note that $0$ is a regular value of $K$. We define the Lagrangian $L(x,\alpha,\beta)=f_b(x)-\alpha \cdot G_b(x) - \beta\cdot H(x)$. By the Lagrange multipliers theorem, if we have a local extremum of $f_b|_{K^{-1}(0)}$ at $y$, then there exist $\lambda \in \mathbb{R}^2, \mu \in\mathbb{R}^j$ such that $\nabla L(y,\lambda,\mu)=0$. This is a system of $b+2+j$ linear equations.
The first one is
\begin{align*}
&1-2\lambda_1 y_1 -\lambda_2 -\mu_{n(1)}=0& \text{ if }1\in m_j\\
&1-2\lambda_1 y_1 -\lambda_2=0& \text{ if }1\not\in m_j
\end{align*}
The following $b-2$ are for $1<i<b$
\begin{align*}
&-2\lambda_1 y_i -\lambda_2 +\mu_{n(i-1)}=0 & \text{ if } i-1 \in m_j \text{ and } i\not \in m_j\\
&-2\lambda_1 y_i -\lambda_2 +\mu_{n(i-1)} - \mu_{n(i)}=0 & \text{ if } i-1 \in m_j \text{ and } i \in m_j\\
&-2\lambda_1 y_i -\lambda_2 - \mu_{n(i)}=0 & \text{ if } i-1 \not \in m_j \text{ and } i\in m_j\\
&-2\lambda_1 y_i -\lambda_2=0 & \text{ if } i-1 \not \in m_j \text{ and } i\not\in m_j\\
\end{align*}
The $b^{\text{th}}$ equation is
\begin{align*}
&-1-2\lambda_1 y_b -\lambda_2 -\mu_{n(b)}=0& \text{ if }(b-1)\in m_j\\
&-1-2\lambda_1 y_b -\lambda_2=0& \text{ if }(b-1)\not\in m_j
\end{align*}
The two following are
\begin{align*}
\sum_{i=1}^{b}y_i^2=1 \quad \text{and}\quad\sum_{i=1}^{b}y_i=0
\end{align*}
The last $j$ equations are
\begin{align*}
&y_i=y_{i+1}& \text{ if }i \in m_j\\
\end{align*}

We use now that $j<b-2$ and show that if we have a local extremum at $y=(y_1,...,y_b)$, then one of the coefficients of $y$ must be $0$.

By adding the $b$ first equations and using $\sum_{i=1}^{b}x_i=0$, we get $\lambda_2=0$. If $j<b-2$, there exist $r>1$ and $s\geq 1$ such that $y_{r-1}<y_{r}=y_{r+1}=...=y_{r+s-1}<y_{r+s}$. We note that $r>1$ and $r+s-1<b$.
Therefore, by adding the equations $r$ to $r+s-1$, we obtain
\begin{align*}
-2\lambda_1(\sum_{i=r}^{i=r+s-1}y_i)=0 \quad\Rightarrow -2\lambda_1 (s-1)y_{r}=0
\end{align*}
This implies $\lambda_1=0$ or $y_{r}=0$. Let $k$ be the smallest integer $\leq b$ such that $k\not \in m_j$. It exists since $j<b-2$. We have now $y_1=...=y_k<y_{k+1}$. By adding the $k$ first equations, we get
\begin{align*}
&1-2\lambda_1 (\sum_{i=1}^{k}y_i) =0\Rightarrow 1-2\lambda_1 k y_1=0\\
\end{align*}
Therefore, $\lambda_1\not=0\Rightarrow y_{r}=0$.

Let us consider the projection $p: \mathbb{R}^b \rightarrow \mathbb{R}^{b-1}: (x_1,...,x_b) \mapsto (x_1,...,x_{r-1},x_{r+1},...,x_b)$. Since $y_r=0$ for some $r\in \{1,...,b\}$, $f_b(y)=(f_{b-1}\circ p)(y)$. Moreover, $p(y)\in D_{b-1}\cap S_{b-1}$, hence we have 

\[
f_b(y)=f_{b-1}(p(y))\geq \min_{x \in D_{b-1}\cap S_{b-1}}f_{b-1}(x)
\]

This concludes the proof of the proposition because $\min_{x \in D_{b-1}\cap S_{b-1}}f_{b-1}(x)=\frac{\sqrt{b-1}}{\sqrt{\left\lfloor{\frac{b-1}{2}}\right\rfloor}\sqrt{\left\lceil{\frac{b-1}{2}}\right\rceil}}>\frac{\sqrt{b}}{\sqrt{\left\lfloor{\frac{b}{2}}\right\rfloor}\sqrt{\left\lceil{\frac{b}{2}}\right\rceil}}$. So the minimum must be $\frac{\sqrt{b}}{\sqrt{\left\lfloor{\frac{b}{2}}\right\rfloor}\sqrt{\left\lceil{\frac{b}{2}}\right\rceil}}$.
\end{paragraph}
\end{proof}
\begin{proof}[Proof of the sharpness result]
For $b=2$, we already proved in part I that the bound is sharp for all $d_B$.
For $b>2$, we define a family of graphs $\{(H^b)_{d_B}\}_{d_B\in \mathbb{N}}$ such that $\sigma_1((H^b)_{d_B})=\frac{b}{\left\lfloor{\frac{b}{2}}\right\rfloor\left\lceil{\frac{b}{2}}\right\rceil\cdot d_B}+O(\frac{1}{d_B^2})$ as $d_B\rightarrow \infty$.

\begin{defi} [The family of graphs $\{(H^b)_{d_B}\}_{d_B\in \mathbb{N}}$]
If $b$ is even, the family of graphs $\{(H^b)_{d_B}\}_{d_B\in \mathbb{N}}$ is defined as shown in figure \ref{Hpair}, that is, there are $\frac{b}{2}$ boundary vertices on each side of the graph. If $b$ is odd it is defined as shown in figure \ref{Himpair}: there is always a vertex more on one of the two sides. In both cases, the path in the middle increases as $d_B$ increases.

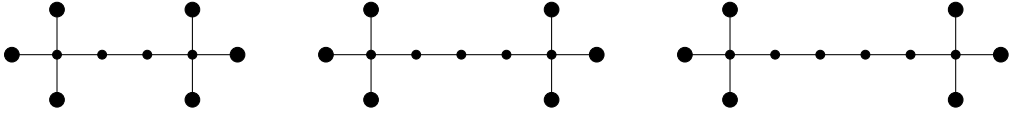
\begin{figure}[h]
\begin{tikzpicture}[scale=0.6]
\draw(1,0)node[sommetb]{};
\draw(0,-1)node[sommetb]{};
\draw(1,-1)node[sommeti]{};
\draw(1,-2)node[sommetb]{};
\draw(2,-1)node[sommeti]{};
\draw(3,-1)node[sommeti]{};
\draw(4,-1)node[sommeti]{};
\draw(5,-1)node[sommetb]{};
\draw(4,0)node[sommetb]{};
\draw(4,-2)node[sommetb]{};
\draw (1,0) -- (1,-1)--(1,-2);
\draw (0,-1)--(1,-1)--(2,-1)--(3,-1)--(4,-1)--(5,-1);
\draw (4,0)--(4,-1)--(4,-2);
\end{tikzpicture}
\quad\quad
\begin{tikzpicture}[scale=0.6]
\draw(0,0)node[sommetb]{};
\draw(-1,-1)node[sommetb]{};
\draw(0,-1)node[sommeti]{};
\draw(0,-2)node[sommetb]{};
\draw(1,-1)node[sommeti]{};
\draw(2,-1)node[sommeti]{};
\draw(3,-1)node[sommeti]{};
\draw(4,-1)node[sommeti]{};
\draw(5,-1)node[sommetb]{};
\draw(4,0)node[sommetb]{};
\draw(4,-2)node[sommetb]{};
\draw (0,0) -- (0,-1)--(0,-2);
\draw (0,-1)--(1,-1)--(2,-1)--(3,-1)--(4,-1)--(5,-1);
\draw (-1,-1)--(0,-1);
\draw (4,0)--(4,-1)--(4,-2);
\end{tikzpicture}
\quad\quad
\begin{tikzpicture}[scale=0.6]
\draw(-1,0)node[sommetb]{};
\draw(-2,-1)node[sommetb]{};
\draw(-1,-1)node[sommeti]{};
\draw(-1,-2)node[sommetb]{};
\draw(0,-1)node[sommeti]{};
\draw(1,-1)node[sommeti]{};
\draw(2,-1)node[sommeti]{};
\draw(3,-1)node[sommeti]{};
\draw(4,-1)node[sommeti]{};
\draw(5,-1)node[sommetb]{};
\draw(4,0)node[sommetb]{};
\draw(4,-2)node[sommetb]{};
\draw (-1,0) -- (-1,-1)--(-1,-2);
\draw (-2,-1)--(-1,-1)--(0,-1)--(1,-1)--(2,-1)--(3,-1)--(4,-1)--(5,-1);
\draw (4,0)--(4,-1)--(4,-2);
\end{tikzpicture}
\caption{$(H^6)_5$, $(H^6)_6$ and $(H^6)_7$}
\label{Hpair}
\end{figure}

\begin{figure}[h]
\begin{tikzpicture}[scale=0.6]
\draw(1,0)node[sommetb]{};
\draw(0,-1)node[sommetb]{};
\draw(1,-2)node[sommetb]{};
\draw(1,-1)node[sommeti]{};
\draw(2,-1)node[sommeti]{};
\draw(3,-1)node[sommeti]{};
\draw(4,-1)node[sommeti]{};
\draw(4.87,-1.5)node[sommetb]{};
\draw(4.87,-0.5)node[sommetb]{};
\draw(4,0)node[sommetb]{};
\draw(4,-2)node[sommetb]{};
\draw (1,0) -- (1,-1)--(1,-2);
\draw (0,-1)--(1,-1)--(2,-1)--(3,-1)--(4,-1);
\draw (4,0)--(4,-1)--(4,-2);
\draw(4.87,-0.5)--(4,-1);
\draw(4.87,-1.5)--(4,-1);
\end{tikzpicture}
\quad\quad
\begin{tikzpicture}[scale=0.6]
\draw(0,0)node[sommetb]{};
\draw(-1,-1)node[sommetb]{};
\draw(0,-1)node[sommeti]{};
\draw(0,-2)node[sommetb]{};
\draw(1,-1)node[sommeti]{};
\draw(2,-1)node[sommeti]{};
\draw(3,-1)node[sommeti]{};
\draw(4,-1)node[sommeti]{};
\draw(4.87,-1.5)node[sommetb]{};
\draw(4.87,-0.5)node[sommetb]{};
\draw(4,0)node[sommetb]{};
\draw(4,-2)node[sommetb]{};
\draw (0,0) -- (0,-1)--(0,-2);
\draw (0,-1)--(1,-1)--(2,-1)--(3,-1)--(4,-1);
\draw (-1,-1)--(0,-1);
\draw (4,0)--(4,-1)--(4,-2);
\draw(4.87,-0.5)--(4,-1);
\draw(4.87,-1.5)--(4,-1);
\end{tikzpicture}
\quad\quad
\begin{tikzpicture}[scale=0.6]
\draw(-1,0)node[sommetb]{};
\draw(-2,-1)node[sommetb]{};
\draw(-1,-1)node[sommeti]{};
\draw(-1,-2)node[sommetb]{};
\draw(0,-1)node[sommeti]{};
\draw(1,-1)node[sommeti]{};
\draw(2,-1)node[sommeti]{};
\draw(3,-1)node[sommeti]{};
\draw(4,-1)node[sommeti]{};
\draw(4.87,-1.5)node[sommetb]{};
\draw(4.87,-0.5)node[sommetb]{};
\draw(4,0)node[sommetb]{};
\draw(4,-2)node[sommetb]{};
\draw (-1,0) -- (-1,-1)--(-1,-2);
\draw (-2,-1)--(-1,-1)--(0,-1)--(1,-1)--(2,-1)--(3,-1)--(4,-1);
\draw (-1,-1)--(0,-1);
\draw (4,0)--(4,-1)--(4,-2);
\draw(4.87,-0.5)--(4,-1);
\draw(4.87,-1.5)--(4,-1);
\end{tikzpicture}
\caption{$(H^7)_5$, $(H^7)_6$ and $(H^7)_7$}
\label{Himpair}
\end{figure}
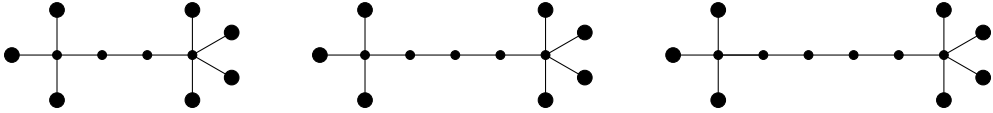

\end{defi}

\begin{lem}
\[
\sigma_1((H^b)_{d_B})=\frac{b}{\left\lfloor{\frac{b}{2}}\right\rfloor\left\lceil{\frac{b}{2}}\right\rceil(d_B-2)+b}
\]
\end{lem}
\begin{proof}
For the convenience of the proof, we will assume that $b$ is even. In this case, the lemma can be rewritten $\sigma_1((H^b)_{d_B})=\frac{4}{b(d_B-2)+4}$. If $b$ is odd, the proof is the same but because of the loss of symmetry the calculations are less pleasant.
 
We number the vertices of the boundary on the left side from $1$ to $\frac{b}{2}$ and those on the right side from $\frac{b}{2}+1$ to $b$. We call $l$ the vertex of the interior connected to the left boundary vertices and $r$ the one connected to the right boundary vertices.

Suppose $v$ is an eigenfunction for $\sigma_1((H^b)_{d_B})$. The coefficients of $v$ corresponding
to the $b$ vertices of the boundary are $v_1,...,v_b$ and those corresponding to $l$ and $r$ are $v_l$ and $v_r$. Assuming that $\sigma_1\not =1$, we have
\begin{align*}
&v_i-v_{l}=\sigma_1 v_i \quad\text{ if } i\leq\frac{b}{2}& \Rightarrow v_1=v_2=...=v_{\frac{b}{2}}=\frac{v_{l}}{1-\sigma_1}\\
&v_i-v_{r}=\sigma_1 v_i \quad\text{ if } \frac{b}{2}\leq i\leq b &\Rightarrow v_{\frac{b}{2}}=...=v_b=\frac{v_{r}}{1-\sigma_1}
\end{align*}
Moreover, using that $\sum_{i=1}^{b}v_i=0$ and $\sum_{i=1}^{b}v_i^2=1$, we get $v_1=...=v_{\frac{b}{2}}=\frac{1}{\sqrt{b}}$ and $v_{\frac{b}{2}}=...=v_b=-\frac{1}{\sqrt{b}}$. Note here that if $b$ were odd, the value on one side were not exactly the opposite of the value on the other side.

Since $v$ is harmonic on the interior, its energy on the path of length $d_B-2$ between the vertices $l$ and $r$ is $\frac{(v_l-v_r)^2}{d_B-2}=\frac{(2v_l)^2}{d_B-2}$.  Thus, the  total energy of $v$ on the graph, $\sum_{i\sim j}(v_i-v_j)^2$, can be written as a function depending only on the variable $v_l$
\[
Q(v_l):=\sum_{i\sim j}(v_i-v_j)^2=\frac{b}{2}(\frac{1}{\sqrt{b}}-v_l)^2 + \frac{(2v_l)^2}{d_B-2} + \frac{b}{2}(v_l - \frac{1}{\sqrt{b}})^2
\]
From equation (\ref{carvar}), we know that $v_j$ must be the point where $Q$ reaches its minimum. After calculation we find $v_l=\frac{\sqrt{b}(d_B-2)}{b(d_B-2)+4}$. Therefore
\[
\sigma_1((H^b)_{d_B})=Q(v_l)=\frac{4}{b(d_B-2)+4}
\]

\end{proof}
Eventually, we obtain by calculation
\[
\sigma_1((H^b)_{d_B})=\frac{b}{\left\lfloor{\frac{b}{2}}\right\rfloor\left\lceil{\frac{b}{2}}\right\rceil(d_B-2)+b}=\frac{b}{\left\lfloor{\frac{b}{2}}\right\rfloor\left\lceil{\frac{b}{2}}\right\rceil\cdot d_B}+O(\frac{1}{d_B^2}))
\]

\end{proof}
This completes the proof of Theorem \ref{thm2}.

\begin{note} In analogy to the Riemannian case, Faber-Krahn type inequalities are also interesting on graphs (see \cite{biyikoglu2007laplacian}). For the Steklov problem, if $b=2$, the graphs $(H^2)_{d_b}$ (which are the paths $(P_n,B_{P_n})$ in the proof of Theorem \ref{thm1}) are minimizers for $\sigma_1$ among all graphs with two vertices in the boundary and the diameter of the bondary equal to $d_B$. We already noted in example \ref{exT} that they are not unique. If $b>2$, we conjecture that $(H^b)_{d_B}$ are minimizers under the same conditions.
\end{note}

\section{Weighted graphs}

We state finally an equivalent of Theorem \ref{thm1} in the context of the Dirichlet-to-Neumann operator on weighted graphs as defined in \cite{hua2017}.

Let $(\Gamma, B)$ be a graph with boundary (we do not more assume that $\Gamma=(V,E)$ is simple).

Let $\mu$ be a symmetric weigth function given by
\begin{align*}
\mu: V \times V& \rightarrow [0,\infty)\\
(i,j)&\mapsto \mu_{ij}=\mu_{ji}
\end{align*}
with $\mu_{ij}=0$ if $\{ij\}\not \in E$. We recall that in the definition of a graph with boundary there is no edge between two vertices of the boundary.

Define the measure on $V$, $m: V\rightarrow (0,\infty)$ as follows:
\[
m_i=\sum_{i\sim j} \mu_{ij}
\]

We will denote by $(\Gamma, B,\mu)$ a weighted graph with boundary.

Recall the variational characterization of the first non-zero eigenvalue for the Dirichlet-to-Neumann operator in this setting according to the definition of Hua et al. (see \cite{hua2017} p.17)
\begin{align}
\sigma_1=\min_{v\in \mathbb{R}^{b}}\left\{\frac{\sum_{i\sim j}\mu_{ij}(v_i-v_j)^2}{\sum_{i\in B}v_i^2m_i}, \sum_{i\in B}v_i m_i=0\right\}
\label{carvarW}
\end{align}

We keep the definition of the distance between two vertices and the one of the diameter of the boundary unchanged. We define $Vol(B):=\sum_{i\in B}m_i$.

\begin{prop}
Let $(\Gamma,B,\mu)$ be a connected weighted graph with diameter of the boundary $d_{B}$. We have
\[
\sigma_1\geq\frac{c}{d_B\cdot Vol(B)}
\]
where $c:=\min_{i\sim j}\mu_{ij}$.
\end{prop}
\begin{note}
This result is analog to the the estimate given for the Laplacian in Theorem 3.5. of  \cite{grigoryan2009analysis}; the proof is also analog but we use the diameter of the boundary instead of using the diameter of the graph. We recall it here.
\end{note}
\begin{proof}[Proof of the proposition] Suppose $v$ is an eigenfunction achieving $\sigma_1$ in (\ref{carvarW}). Let $\alpha$ be the vertex where $|v_{\alpha}|=\max_{i\in B}|v_i|$. Since $\sum_{i\in B}v_i m_i=0$ and $m_i>0$, there exists a vertex $\beta \in B$ satisfying $v_{\alpha}v_{\beta}<0$. Let $P$ denote a shortest path in $\Gamma$ joining $\alpha$ and $\beta$. Then, by (\ref{carvarW}) and using Cauchy-Schwarz inequality, we have
\begin{align*}
\sigma_1&=\frac{\sum_{i\sim j}\mu_{ij}(v_i-v_j)^2}{\sum_{i\in B}v_i^2m_i}
\geq \frac{\sum_{\{ij\}\in P}\mu_{ij}(v_i-v_j)^2}{Vol(B)\cdot v_{\alpha}^2}\\
&\geq \frac{c\sum_{\{ij\}\in P}(v_i-v_j)^2}{Vol(B)\cdot v_{\alpha}^2}\geq\frac{c(v_{\alpha}-v_{\beta})^2}{d_B\cdot Vol(B)\cdot v_{\alpha}^2}\\
&\geq \frac{c}{d_b\cdot Vol(B)}
\end{align*}
with $c:=\min_{i\sim j}\mu_{ij}$.
\end{proof}

\bibliography{biblio}{}
\bibliographystyle{plain}

\bigskip
\footnotesize

\textsc{Université de Neuchâtel, Institut de mathématiques, Rue Emile-Argand 11, Case postale 158, 2009 Neuchâtel, Switzerland}\par\nopagebreak
\textit{E-mail address}, \texttt{helene.perrin@unine.ch}

\end{document}